\documentclass{article}
\usepackage{amsmath,amssymb,latexsym,amsthm,mathrsfs}

\newtheorem{theo}{Theorem}[section]
\newtheorem{pro}[theo]{Proposition}
\newtheorem{lem}[theo]{Lemma}
\newtheorem{cor}[theo]{Corollary}

\newcommand{\ra}{\rightarrow}
\theoremstyle{definition}
\newtheorem{defin}[theo]{Definition}

\newtheorem{exa}{Example}[section]

\theoremstyle{remark}
\newtheorem{rem}[theo]{Remark}

\begin{document}

\title{On some random walks driven by spread-out measures}
\author{Laurent Saloff-Coste\thanks{%
Both authors partially supported by NSF grant DMS 1004771} \\
{\small Department of Mathematics}\\
{\small Cornell University} \and Tianyi Zheng \\
{\small Department of Mathematics}\\
{\small Stanford  University} }
\maketitle

\begin{abstract}
Let $G$ be a finitely generated group equipped with a symmetric generating $%
k $-tuple $S$. Let $|\cdot|$ and $V$ be the associated word length and
volume growth function. Let $\nu$ be a probability measure such that $%
\nu(g)\simeq [(1+|g|)^2V(|g|)]^{-1}$. We prove that if $G$ has polynomial
volume growth then $\nu^{(n)}(e) \simeq V(\sqrt{n\log n})^{-1}$. We also
obtain assorted estimates for other spread-out probability measures.
\end{abstract}

\noindent{\bf 2010 MSC:} 20F65, 60J10, 60J51 \vspace{.1in}

\noindent{\bf Key words:} Random walk, polynomial volume growth,
local limit theorems

\section{Introduction}

\setcounter{equation}{0}

This work is concerned with questions related to a number of recent studies
where ``stable-like'' processes and random walks are considered. We
focus on random walks on groups, mostly nilpotent groups and groups of
polynomial volume growth, associated with various type of spread-out
probability measures. Here, spread-out is used in a non-technical sense 
to convey the idea that these
measures do not have finite support.

Given a probability measure $\nu$ on a (finitely generated) group $G$,
we consider the discrete time random walk $(X_n)_0^\infty$ driven by $\nu$
and started at $X_0=e$. This means that $X_n=\xi_1\dots\xi_n$, $n\ge 1$, 
where $(\xi)_1^\infty$ is a i.i.d.\ sequence of $G$-random variables 
with common law $\nu$.
The distribution of $X_n$ is the convolution power 
$\nu^{(n)}$.  We also consider the associated continuous time random walk
$X_t$ whose distribution is given by 
$$p_t(g)=e^{-t}\sum_0^\infty \frac{t^n}{n!}\nu^{(n)}(g).$$
This continuous time process will serve as a tool in the study 
of the discrete random walk driven by $\nu$, a technique that 
has been used by many authors before.

The question addressed in the present work is the following. Assuming good upper 
bounds on $\mu^{(n)}(e)$, under which circumstances can one prove matching 
lower bounds? Further, can one describe (in a certain sense)
the region  where, for a given $n$, $\nu^{(n)}(g)\simeq \nu^{(n)}(e)$? 
We provide answers for measures $\nu$ that are quite natural and for 
which well understood existing techniques are insufficient and/or 
need to be modified.

\subsection{Main definitions}

\begin{defin}
\label{norm} We say that $\|\cdot\|: G\rightarrow [0,\infty)$ is a norm on $%
G $ if $\|g\|=0$ if and only if $g=e$ and, for all $g,h\in G$, $\|gh\|\le
\|g\|+\|h\|$. Given a norm $\|\cdot\|$, we say that $V(r)=\#\{g\in G:
\|g\|\le r\}$ is the associated volume function.
\end{defin}

The simplest and most common example is of a norm is provided by the 
word-length associated to a given finite symmetric set of generators. We 
will encounter other norms as well. 

The properties studied in this work  are the following.

\begin{defin}
\label{controlled} Let $\mu$ be a symmetric probability measure on a group $%
G $. Let $\|\cdot\|$ be a norm with volume function $V$. Let $r:
(0,\infty)\rightarrow (0,\infty),\; t\mapsto r(t)$, be a non-decreasing
function. Let $(X_n)_0^\infty$ be the random walk on $G$ driven by $\mu$%
. We say that $\mu$ is $(\|\cdot\|,r)$-controlled if the following
properties are satisfied:

\begin{enumerate}
\item For all $n$, $\mu^{(2n)}(e)\simeq V(r(n))^{-1}.$

\item For all $\epsilon >0$ there exists $\gamma\in (0,\infty)$ such that 
\begin{equation*}
\mathbf{P}_e\left(\sup_{0\le k\le n}\{\|X_k\|\}\ge \gamma r(n)\right)\le
\epsilon .
\end{equation*}
\end{enumerate}
\end{defin}
The first of these two properties is rather straightforward and 
self-explanatory. It provides a two-sided estimate for the probability of 
return of the random walk. In more general contexts, this property is also 
known as a two-sided ``on-diagonal'' bound. The second property 
is related to the first in so far as it actually easily implies the lower bound 
$\mu^{(2n)}(e)\ge V(cr(n))^{-1}$. It also provides a weak control 
of the behavior of  $\mu^{(n)}(g)$ away from the neutral element $e$.

\begin{defin}
\label{controlled*} Let $\mu$ be a symmetric probability measure on a group $%
G$. Let $\|\cdot\|$ be a norm with volume function $V$. Let $r:
(0,\infty)\rightarrow (0,\infty),\; t\mapsto r(t)$, be an increasing continuous
function with inverse $\rho$. Let $(X_n)_0^\infty$ denote the random walk on 
$G$ driven by $\mu$. We say that $\mu$ is strongly $(\|\cdot\|,r)$%
-controlled if the following properties are satisfied:

\begin{enumerate}
\item There exists $C\in (0,\infty)$ and, for any $\kappa>0$, there exists $%
c(\kappa)>0$ such that, for all $n\ge 1$ and $g$ with $\|g\|\le \kappa r(n)$%
, 
\begin{equation*}
c(\kappa)V(r(n)))^{-1} \le \mu^{(2n)}(g)\le C V(r(n))^{-1}.
\end{equation*}

\item There exists $\epsilon, \gamma_1, \gamma_2\in (0,\infty)$, $%
\gamma_2\ge 1$, such that, for all $n,\tau$ such that $\frac{1}{2}%
\rho(\tau/\gamma_1)\le n \le \rho(\tau/\gamma_1) $ 
\begin{equation}  \label{strong}
\inf_{x:\|x\|\le \tau }\left\{\mathbf{P}_x\left(\sup_{0\le k\le n}
\{\|X_k\|\}\le \gamma_2 \tau; \|X_n\|\le \tau \right)\right\}\ge \epsilon .
\end{equation}
\end{enumerate}
\end{defin}

Strong control implies the following useful estimate. The last section 
of this paper gives an application of this estimate to random walks 
on wreath products.

\begin{pro}
\label{pro-StC} Assume that $r$ is continuous increasing with inverse $\rho$
and that the symmetric probability measure $\mu$ is strongly $(\|\cdot\|, r)$%
-controlled. Then, for any $n$ and $\tau$ such that $\gamma_1 r(2n)\ge \tau$%
, we have 
\begin{equation}  \label{strong*}
\inf_{x:\|x\|\le \tau}\left\{\mathbf{P}_x\left(\sup_{0\le k\le
n}\{\|X_k\|\le \gamma_2 \tau; \|X_n\|\le \tau\right)\right\}\ge \epsilon^{1+
2n/\rho(\tau/\gamma_1) }.
\end{equation}
\end{pro}

\begin{proof} By induction on $\ell\ge 1$ such that
$1\le   2n/\rho( \tau /\gamma_1) < (\ell+1)$, we are going to prove that
$$
\inf_{x:\|x\|\le \tau}\left\{\mathbf 
P_x\left(\sup_{0\le k\le n}\{\|X_k\|\le \gamma_2 \tau; \|X_n\|\le  
\tau\right)\right\}\ge \epsilon^{1+ \ell }.$$
This easily yields the desired result. For $\ell=1$, the inequality 
follows from the strong control assumption. 
Assume the property holds for some $\ell\ge 1$. Let $n,\tau$ be such that
$(\ell+1)\le 2n/\rho(\tau/\gamma_1) < (\ell+2)$. 
Choose $n'$ such that 
$n-n'= \lceil \rho(\tau/\gamma_1)/2\rceil $
and note that $ 2 n'\in [1, (\ell+1)\rho(\tau/\gamma_1))$.
Write  $Z_{n}= \sup_{k\le n}\{\|X_k\|\}$ and, for any $x$ 
such that $\|x\|\le \tau$,
\begin{eqnarray*}
\lefteqn{\mathbf 
P_x\left( Z_n\le \gamma_2 \tau;\|X_n\|\le \tau\right)}&&\\
&\ge &
\mathbf P_x\left( Z_n\le \gamma_2 \tau;\|X_{n'}\|\le\tau;
 \|X_n\|\le \tau\right)\\
&\ge & \mathbf P_x\left( Z_{n'}\le \gamma_2 \tau;\|X_{n'}\|\le\tau;
\sup_{n'\le k\le n}\{\|X_k\|\}\le \gamma_2 \tau;
 \|X_n\|\le \tau\right)\\
&=&
\mathbf E_x\left( \mathbf 1_{\{ Z_{n'}\le \gamma_2 \tau;\|X_{n'}\|\le\tau\}}\mathbf P_{X_{n'}}\left( Z_{n-n'}\le \gamma_2 \tau;
 \|X_{n-n'}\|\le \tau\right)\right)\\
&\ge &
\epsilon 
\mathbf P_x\left(Z_{n'}\le \gamma_2 \tau;\|X_{n'}\|\le\tau\right)
\ge \epsilon ^{2+\ell}.
\end{eqnarray*}
This gives the desired property for $\ell+1$.
\end{proof}

\subsection{Word-length radial measures}

Let the group $G$ be equipped with a generating $k$-tuple 
\begin{equation*}
S=(s_1,\dots, s_k)
\end{equation*}
and the associated finite symmetric set of generators $\mathcal{S}=\{s^{\pm
1}_1,\dots, s^{\pm 1}_k\}$. Let $|g|$ be the associated word length, that
is, the minimal $k$ such that $g= u_1\dots u_k$ with $u_i\in \mathcal S$, $1\le i\le
k $. By definition, the identity element $e$ has length $0$. Hence, $|\cdot|$
is a norm and $(x,y)\rightarrow |x^{-1}y|$ is a left-invariant distance
function on $G$. Let 
\begin{equation*}
V_S(r)= \#\{G: |g|\le r\}
\end{equation*}
be the volume of the ball of radius $r$. We say that $G$ has polynomial
volume growth of degree $D$ if  $%
V_S(r)\simeq r^D$ in the sense that the ratio $V_S(r)/r^D$ is bounded away
from $0$ and $\infty$ for $r\ge 1$. Finitely generated nilpotent groups have
polynomial volume growth and, by Gromov's theorem, any finitely generated
group with polynomial volume growth contains a nilpotent subgroup of finite
index. More precisely, any finitely generated group $G$ such that 
there exist constant s $C,A$ and a sequence $n_k$ with $V(n_k)\le Cn_k^A$
contains a nilpotent subgroup of finite
index and thus has polynomial volume growth of degree $D$ for some integer $D$.
See, e.g., \cite{delaH}.
\begin{exa}
Let $G$ be equipped with a word-length function $|\cdot|$ associated with a
symmetric finite generating subset. Assume that $G$ has polynomial volume
growth. The main results of \cite{HSC} imply that, for any symmetric
probability measure $\mu$ with finite generating support, $\mu$ is strongly $%
(|\cdot|,t\mapsto\sqrt{t})$-controlled. The main results of \cite{BBK,BGK}
show that, if $\nu_\beta$ is symmetric and satisfies $\nu_\beta(g)\simeq
[(1+|g|)^{\beta}V(|g|)]^{-1}$ with $\beta\in (0,2)$, then $\nu_\beta$ is
strongly $(|\cdot|,t\mapsto t^{1/\beta})$-controlled. See also \cite%
{BSClmrw,MSC}.
\end{exa}

One example that motivates the  present work is the case of the measure
$$\nu_2(g)= \frac{c}{(1+|g|)^2V(|g|)}.$$
Can one provides good estimates for $\nu_2^{(n)}(e)$ on groups 
of polynomial volume growth? The following theorem gives a very 
satisfactory answer to this question and covers not only this particular 
example but the full range of cases passing through  the classical threshold 
corresponding to the second moment condition. 
\begin{theo}\label{th-phi}
Let $G$ be equipped with a word-length function $|\cdot|$ associated with a
symmetric finite generating subset. Let $V$ be the associated volume function 
and assume that $G$ has polynomial volume
growth.  Let $\phi:[0,\infty)\ra [1,\infty)$ be a continuous regularly varying function of positive index. Let $r$ be the inverse function of 
$$t\mapsto  t^2/\int_0^t \frac{sds}{\phi(s)}.$$
Let $\nu_\phi$ be a symmetric probability measure such that 
\begin{equation}\label{nuphi}
\nu_\phi(g)\simeq \frac{1}{\phi(|g|)V(|g|)}.
\end{equation}
Then $\nu_\phi$ is strongly $(|\cdot|,r)$-controlled.
\end{theo}

\begin{exa} 
Assume that $\phi(t)=(1+t)^\beta \ell(t)$ with $\ell$ positive 
continuous and slowly varying (we refer the reader to \cite[Chap. I]{BGT}
for the definition and basic properties of slowly and regularly varying 
functions. The scaling function $r$ of Theorem \ref{th-phi}
can be described more explicitly as follows.
\begin{itemize}
\item If $\beta>2$, $r(t)\simeq t^{1/2}$.
\item If $\beta < 2$, we
have $t^2/\int_0^t \frac{sds}{\phi(s)} \simeq  c_\phi\,
 \phi(t)$ and 
$r$ is essentially the inverse of $\phi$, namely,
$$r(t)\simeq t^{1/\beta} \ell_{\#}^{1/\beta}(t^\beta) $$
where $\ell_{\#}$ is the de Bruijn conjugate of $\ell$. 
See \cite[Prop.\ 1.5.15]{BGT}. For example, if $\ell$ has the property that 
$\ell(t^a)\simeq \ell(t)$ for all $a>0$ then $\ell_{\#}\simeq 1/\ell$.
\item The case $\beta=2$ is more subtle and the proof is more difficult. 
The function  
$\psi:t\mapsto \int_0^t\frac{sds}{\phi(s)}$ is slowly varying and satisfies
$\psi(t)\ge \frac{c_1}{\ell(t)}$. For example, 
if $\ell\equiv 1$, we have $\psi(t)\simeq \log t$
and $r(t)\simeq (t\log t)^{1/2}$. When 
$\ell(t)=(\log t)^\gamma$ with $\gamma\in \mathbb R$ then
\begin{itemize}
\item If $\gamma>1$,  $\psi(t)\simeq 1$ and  $r(t)\simeq  t^{1/2}$;
\item If $\gamma=1$,  $\psi(t)\simeq \log \log t$ and  $r(t)\simeq  
(t\log\log t)^{1/2}$;
\item If $\gamma <1$,  $\psi(t)\simeq (\log t)^{1-\gamma}$ and  
$r(t)\simeq  (t (\log t)^{1-\gamma})^{1/2}$;
\end{itemize}\end{itemize}
\end{exa}

\subsection{Measures supported by the powers of the generators}

In the critical case when $\phi$ is regularly varying of index $2$ and 
$\nu_\phi$ has infinite second moment (i.e., $\sum |g|^2\nu_\phi(g)=\infty$), 
the proof of Theorem \ref{th-phi} 
makes essential use of some of the results from 
\cite{SCZ-nil} which are related to variations on 
the following  class of examples.
Recall that $G$ is equipped with the generating $k$-tuple $%
S=(s_1,\dots,s_k)$. For any $k$-tuple $a=(\alpha_1,\dots,\alpha_k)\in
(0,\infty)^k$, and consider the probability measure $\mu_{S,a}$ supported on
the powers of the generators $s_1,\dots,s_k$ and defined by 
\begin{equation}  \label{muSa}
\mu_{S,a}(g)=\frac{1}{k}\sum_1^k \sum_{m\in \mathbb{Z}} \frac{\kappa_i}{%
(1+|m|)^{1+\alpha_i}} \mathbf{1}_{s_i^m}(g).
\end{equation}

Set 
\begin{equation*}
\widetilde{\alpha}_i=\min\{\alpha_i,2\} \;\mbox{ and }\;\; \alpha_*=\max\{%
\widetilde{\alpha}_i, 1\le i\le k\}.
\end{equation*}
Define 
\begin{equation}  \label{normSa}
\|g\|_{S,a}= \min\left\{ r: g= \prod _{j=1}^m s_{i_j}^{\epsilon_j}:
\epsilon_j=\pm 1,\;\; \#\{j: i_j =i\}\le r^{\alpha_*/\widetilde{\alpha}%
_i}\right\}.
\end{equation}
Note that $g\mapsto \|g\|_{S,a}: G\rightarrow [0,\infty)$ is a norm.
Consider also the measure 
\begin{equation}  \label{nuSa}
\nu_{S,a,\beta}(g)= \frac{c(G,a,\beta)}{(1+\|g\|_{S,a})^\beta V_{S,a}
(\|g\|_{S,a})}
\end{equation}
with $\beta\in (0,2)$.

Under the  assumption that $G$ is nilpotent and $\{s_i: \alpha_i\in
(0,2)\}$ generates a subgroup of finite index in $G$, it is proved in \cite%
{SCZ-nil} that there exists a positive real $D_{S,a}$ such that 
\begin{equation*}
Q_{S,a}(r)=\#\{\|g\|_{S,a}\le r^{1/\alpha_*}\}\simeq r^{D_{S,a}}
\end{equation*}
and 
\begin{equation*}
\mu_a^{(n)}(e)\le C_{S,a} n ^{-D_{S,a}},\;\;\nu_{S,a,\beta}^{(n)}(e)\le
C_{S,a,\beta}n^{-\alpha_*D(S,a)/\beta}.
\end{equation*}
Here we prove the following complementary result.

\begin{theo}
\label{th-muSa} Let $G$ be a finitely generated nilpotent group equipped
with a generating $k$-tuple $S=(s_1,\dots,s_k)$. Referring to the notation
introduced above, fix $a\in (0,\infty)^k$ and assume that $\{s_i:
\alpha_i\in (0,2)\}$ generates a subgroup of finite index in $G$.

\begin{itemize}
\item The probability measure $\mu_{S,a}$ is strongly $(\|\cdot\|_{S,a},t%
\mapsto t^{1/\alpha_*})$-controlled.

\item For any $\beta\in (0,2)$, $\nu_{S,a,\beta}$ is strongly $%
(\|\cdot\|_{S,a},t\mapsto t^{1/\beta })$-controlled.
\end{itemize}
\end{theo}

\begin{rem}
In \cite{SCZ-nil}, a detailed analysis of the sub-additive function $%
\|\cdot\|_{S,a}$ and the associated geometry is given. This analysis is key
to the above result and to its proper understanding. For instance, it is
important to understand that the parameter $\alpha_*$ is not necessarily a
significant
 parameter. It is the quantity $\|\cdot\|_{S,a}^{\alpha_*}$ that is the
important expression. Indeed, for any given nilpotent group $G$, \cite{SCZ-nil}
describes conditions on two pairs of tuples $(S,a)$, $(S^{\prime },a^{\prime
})$, 
\begin{equation*}
S=(s_i)_1^k\in G^k,a=(\alpha_i)_1^k\in (0,\infty)^k, S^{\prime }=(s^{\prime
}_i)_1^{k^{\prime }}\in G^{k^{\prime }}, a^{\prime }=(\alpha^{\prime
}_i)_1^{k^{\prime }}\in (0,\infty)^{k^{\prime }},
\end{equation*}
such that $\|\cdot\|_{S,a}^{\alpha_*}\simeq \|\cdot\|_{S^{\prime },a^{\prime
}}^{\alpha^{\prime }_*}.$ Since the geometry $\|g\|_{S,a}$ is studied and
described rather explicitly in \cite{SCZ-nil}, the above results give
rather concrete controls of the random walks driven with $\mu_{S,a}$ or $%
\nu_{S,a,\beta}$.
\end{rem}

On the one hand, in the case of the measures $\nu_{S,a,\beta}$
and with much more work, 
it is possible to improve upon the statement of Theorem \ref{th-muSa}
and obtain two side  pointiwse bound on $\nu_{S,a,\beta}$. Indeed,
based on the results of \cite{SCZ-nil}, it is proved in \cite{MSC}
that, for all $g\in G$ and $n\ge 1$, 
\begin{equation*}
\nu^{(n)}_{S,a,\beta}(g)\simeq \frac{ n}{(n^{1/\beta}+\|g\|_{S,a})^{%
\alpha_*D_{S,a}+\beta}} \simeq \min\left\{\frac{1}{n^{\alpha_* D_{S,a}/\beta}%
},\frac{n}{\|g\|_{S,a}^{\alpha_*D_{S,a}+\beta}}\right\}.
\end{equation*}

On the other hand, in the case of the measures $\mu_{S,a}$, 
Theorem \ref{th-muSa} provides the most detailed result available at this time. 
Indeed, available techniques do not seem to be adequate to provide a sharp 
two-sided bound for $\mu_{S,a}^{(n)}(g)$.

\subsection{A short guide}
Section \ref{sec-Dav} is based on well-known variations of the 
celebrated Davies off-diagonal upper bound technique. Our key observation is 
that, even in cases where we do not expect to obtain full off-diagonal 
upper bounds, Davies technique provides enough information to prove 
{\em control} in the sense of Definition \ref{controlled}. 

Section \ref{sec-psc} describes the notion of pointwise pseudo-Poincar\'e 
inequality (a variation on the idea introduced in \cite{CSCiso}) and shows 
how, with the help of the underlying group structure, 
a pseudo-Poincar\'e inequality
allow us to upgrade {\em control} to {\em strong control}.

Section \ref{sec-pow} applies the earlier results to a family 
of probability measures and random walks introduced in \cite{SCZ-nil}. These 
measures are supported on the powers of the given generators. They provide 
examples for which no good off-diagonal upper bounds are known at this time. 
Nevertheless, the results developed here apply and capture useful properties 
of the associated random walks.

Section \ref{sec-rad} is concerned with radial type measures where radial 
refers to a given norm on the group $G$. The simplest 
and most interesting case is when this norm is taken to be the usual 
word-length associated with a finite symmetric set of generators and, in 
this case,  we  prove Theorem \ref{th-phi}.

Section \ref{sec-wp} describes the applications to a class of random walks on 
wreath products. The notion of strong control 
(on the base group of the wreath product) leads to lower bounds 
on the probability of return on the wreath product.   

\section{Davies method, tightness and control}  \label{sec-Dav}

\setcounter{equation}{0}

\subsection{Davies method for the truncated process}

In this section, we review how Davies' method applies to the continuous time
process associated with truncated jumping kernels. We follow \cite[Section 5]%
{Mim} rather closely even so our setup is somewhat different. The first paper 
treating jump kernels by Davies method is \cite{CKS}.

Throughout this section $G$ is a discrete group equipped with its counting
measure. Fix a norm $g\mapsto \|g\|$ with volume function $V$ and set $%
d(x,y)=\|x^{-1}y\|$. Note that $d$ is a distance function on $G$. Consider
the left-invariant symmetric jumping kernel 
\begin{equation*}
J(x,y)= \nu(x^{-1}y)
\end{equation*}
associated to a given symmetric probability measure $\nu$ on $G$. For $R>0$,
define 
\begin{equation}  \label{def-dG}
\delta _{R}:=\sum_{\left\Vert x\right\Vert >R}\nu (x) \mbox{ and }\;\; 
\mathcal{G}(R)=\sum_{\left\Vert x\right\Vert \leq R}\left\Vert x\right\Vert
^{2}\nu (x),
\end{equation}
and 
\begin{equation*}
J_{R}(x,y):=J(x,y)\mathbf{1}_{\{d(x,y)\leq R\}},\;\; J_{R}^{\prime
}(x,y):=J(x,y)\mathbf{1}_{\{d(x,y)>R\}}.
\end{equation*}

Denote by $p(t,x,y)$ and $p_R(t,x,y)$ the transition densities of the
continuous time processes associated to $J$ and $J_R$, respectively. In
particular, 
\begin{equation*}
p(t,x,y)=p_t(x^{-1}y)=e^{-t}\sum_0^\infty \frac{t^n}{n!} \nu^{(n)}(x^{-1}y).
\end{equation*}
Let 
\begin{equation}\label{def-dir}
\mathcal{E}(f,f)= \mathcal{E}_\nu(f,f)= 
\frac{1}{2}\sum_{x,y} (f(x)-f(y))^2J(x,y)
\end{equation}
be the corresponding Dirichlet form  and set also 
\begin{equation*}
\; \mathcal{E}%
_R(f,f)= \frac{1}{2}\sum_{x,y} (f(x)-f(y))^2J_R(x,y).
\end{equation*}
Note that 
\begin{eqnarray*}
\mathcal{E}(f,f)-\mathcal{E}_{R}(f,f) &=&\frac{1}{2} \sum_{x,y:d(x,y)>R}%
\left\vert f(x)-f(y)\right\vert ^{2}J(x,y) \\
&\leq &\sum_{x,y: d(x,y)>R}(f(x)^{2}+f(y)^{2})J(x,y) \leq 2\left\Vert
f\right\Vert _{2}^{2}\delta_R
\end{eqnarray*}%
Consider the on-diagonal upper bound given by%
\begin{equation}  \label{odub}
\forall\,x\in G,\;t>0,\;\;p_t(e)\leq m(t),
\end{equation}
where $m:[0,\infty)\ra 
[0,\infty)$ is continuous regularly varying function of negative index 
at infinity and $m(0)<\infty$. 
Since the function $t\mapsto m(t)$ may present a slowly
varying factor, we follow \cite{Mim}. The starting point is the log-Sobolev
inequality 
\begin{equation}
\sum f^{2}\log f\leq \epsilon \mathcal{E}_{R}(f,f)+(2\epsilon \delta
_{R}+\log m(\epsilon ))\left\Vert f\right\Vert _{2}^{2}+\left\Vert
f\right\Vert _{2}^{2}\log \left\Vert f\right\Vert _{2}  \label{log-Sobolev}
\end{equation}%
with $\epsilon>0$ which follows from (\ref{odub}) by \cite[Theorem 2.2.3]%
{Dav}. The following technical proposition is the key to most of the results
obtained in later sections.

\begin{pro}
\label{Truncated}Assume that the on-diagonal upper bound \emph{(\ref{odub})}
holds with $m$ regularly varying of negative index. Then there is a constant 
$C$ such that, for all $R,t>0$ and $x\in G$ we have 
\begin{equation*}
p_{R}(t,e,x)\leq Ce^{4\delta _{R}t}m(t)\left(\frac{t}{R^{2}/\mathcal{G}(R)}
\right) ^{\left\Vert x\right\Vert /3R}.
\end{equation*}
\end{pro}
\begin{rem} The bound in this proposition is better than the 
uniform bound $p_R(t,e,x)\le Ce^{4\delta_Rt} m(t)$ 
only when $t<R^2/\mathcal G(R)$.
\end{rem}
\begin{proof} It suffices to consider the case $t<R^2/\mathcal G(R)$.
Starting with (\ref{log-Sobolev}), we apply Davies method, as described in
\cite[Section 5.1]{Mim} to estimate $p_{R}(t,e,x)$. Let%
$$\Lambda_R(\psi)= \max\left\{\|e^{-2\psi}\Gamma_R( e^{\psi},e^{\psi})\|_\infty,\|e^{\psi}\Gamma_R(e^{-\psi},e^{-\psi})\|_\infty\right\}$$
with 
$$\Gamma_R(\psi)(x)=\sum_y |\psi(x)-\psi_y)|^2J_R(x,y).$$
then by \cite[Corollary 5.3]{Mim}, 
\[
p_{R}(t,x,y)\leq Cm(t)\exp \left(
4\delta _{R}t +72 \Lambda_{R}(\psi)^2 t- \psi(y)+\psi(x)\right). 
\]%
Consider the case $x=x_0$ and $y=e$.
For $\lambda>0$, set $
\psi (z)=\lambda (\left\Vert x_{0}\right\Vert -\left\Vert z\right\Vert
)^{+} $ and write%
\begin{eqnarray*}
e^{-2\psi (z)}\Gamma _{R}(e^{\psi },e^{\psi })(z)
&=&\sum_y (e^{\psi (z)-\psi (y)}-1)^{2}J _{R}(z,y) \\
&\leq &e^{2\lambda R}\sum_y (\psi (z)-\psi
(y))^{2}J _{R}(z,y) \\
&\leq &\lambda ^{2}e^{2\lambda R}\sum_{\left\Vert y\right\Vert \leq
R}\left\Vert y\right\Vert ^{2}d\nu \leq R^{-2}e^{3\lambda R}\mathcal{G}(R).
\end{eqnarray*}%
Since $\psi (e)=\lambda \left\Vert x_{0}\right\Vert 
$, we obtain%
\[
p_{R}(t,e,x_{0})\leq Cm(t)\exp \left( 4\delta _{R}t+72tR^{-2}e^{3\lambda R}%
\mathcal{G}(R)-\lambda \left\Vert x_{0}\right\Vert \right) . 
\]%
Since  $t<R^2/\mathcal G(R)$, we can set
\[
\lambda =\frac{1}{3R}\log \frac{R^{2}}{t\mathcal{G}(R)} 
\]%
so that the second term $72tR^{-2}e^{3\lambda R}\mathcal{G}(R)$ is a
constant. This yields the stated upper bound.
\end{proof}

\subsection{Control}

Meyer's
construction is a useful technique to construct the process $X_{s}$ by adding
big jumps to $X_{s}^{R}$. See, e.g., \cite{Meyer} and \cite[Lemma 3.1]{BGK}.
In this section, we combine the off-diagonal upper bound in Proposition \ref%
{Truncated} with Meyer's construction to derive control type results for the
process with jumping kernel $J$. Our goal is to show that there is a certain
choice of continuous increasing function $r(t)$, for any $\varepsilon >0$,
there exists constant $\gamma >1$ such that 
\begin{equation*}
\mathbf{P}_e\left( \sup_{s\leq t}\left\Vert X_{s}\right\Vert \ge \gamma
r(t)\right)\le \varepsilon.
\end{equation*}

Let $X_{s}^{R}$ denote the process with truncated kernel $J_{R}.$ 
It follows from Meyer's construction that (see \cite{Meyer} and \cite[Lemma 3.1]{BGK}) 
\begin{equation*}
\mathbf{P}_e(X_{s}\neq X_{s}^{R}\text{ for some }s\leq t)\leq t\delta _{R}.
\end{equation*}%
For any $r>0$, $\gamma >1$, both to be specified later, we have%
\begin{eqnarray}
\lefteqn{ \mathbf P_{e}\left( \sup_{s\leq t}\left\Vert X_{s}\right\Vert \geq
\gamma r\right)}\hspace{.5in}&&  \notag \\
&\leq &\mathbf{P}_e\left( \sup_{s\leq t}\left\Vert X_{s}^{R}\right\Vert \geq
\gamma r\right) + \mathbf{P}_{e}\left( X_{s}\neq X_{s}^{R}\text{ for some }%
s\leq t\right)  \notag \\
&\leq & \mathbf{P}_{e}\left( \sup_{s\leq t}\left\Vert X_{s}^{R}\right\Vert
\geq \gamma r\right) +t\delta _{R}  \notag \\
&\leq &2\sup_{s\leq t}\left\{\mathbf{P}_e\left( \left\Vert
X_{s}^{R}\right\Vert \geq \frac{\gamma }{2}r\right)\right\} +t\delta _{R}
\label{Meyer}
\end{eqnarray}

This will be helpful in deriving the following result.

\begin{pro}
\label{tightness} Assume that for all $\rho>0$, $V(2\rho)\le C_{VD}V(\rho).$
Assume also that $\nu$ is such that \emph{(\ref{odub})} holds where $m$ is
regularly varying of negative index. For $\varepsilon >0$, fix a function $%
R(t)$ such that 
\begin{equation*}
2 t\delta _{R(t)}<\varepsilon \mbox{ and }\;\; \frac{t}{R(t)^{2}/\mathcal{G}%
(R(t))}<e^{-1}.
\end{equation*}
Let $r(t)\ge R(t)$ be a positive continuous increasing function such that 
\begin{equation*}
\sup_{t>0}\left\{m(t)V(r(t))e^{-r(t)/6R(t)}\right\}<\infty .
\end{equation*}
Then, for any $\epsilon > 0$ there exists a constant $\gamma \ge 1$ such
that 
\begin{equation*}
\mathbf{P}_e\left( \sup_{s\leq t}\left\Vert X_{s}\right\Vert \geq \gamma
r(t)\right) <\varepsilon
\end{equation*}
In particular,  we have 
\begin{equation*}
p(t,e,e)\geq \frac{1-\varepsilon }{V(\gamma r(t))}.
\end{equation*}
If, in addition $V(r(t))\simeq m(t)$, then the measure $\nu$ is $(\|\cdot\|, r)$-controlled in continuous time.
\end{pro}

\begin{proof}
Proposition \ref{Truncated} implies that for $s\leq t,$ 
\begin{eqnarray*}
p_{R}(s,e,x) &\leq &Cm(s)\left( \frac{s}{R^{2}/\mathcal{G}(R)}\right)
^{\left\Vert x\right\Vert /3R} \\
&=&Cm(s)\left( \frac{s}{t}\right) ^{\left\Vert x\right\Vert /3R}\left( \frac{%
t}{R^{2}/\mathcal{G}(R)}\right) ^{\left\Vert x\right\Vert /3R}.
\end{eqnarray*}%
Fix $R=R(t)$, $r=r(t)\ge R$,  
decompose $\{x: \|x\|\ge  \frac{\gamma }{2}r\}$ into dyadic
annuli  $\{x: \|x\| \simeq 2^{i}\gamma r\}$ and write 
\begin{eqnarray*}
\lefteqn{\mathbf P_e\left( \left\Vert X_{s}^{R}\right\Vert 
\geq \frac{\gamma r}{2}%
\right)}\hspace{.5in} &&\\
&\leq& C\sum_{i=0}^{\infty }m(s)\left( \frac{s}{t}\right) ^{2^{i-1}\gamma
/3}e^{-2^{i-1}\gamma r/3R}V(2^{i}\gamma r(t)) \\
&=&Cm(t)V(\gamma r)\sum_{i=0}^{\infty }\frac{m(s)}{m(t)}\left( \frac{s}{t}%
\right) ^{2^{i-1}\gamma /3}e^{-2^{i-1}\gamma r/3R}\left( \frac{%
V(2^{i}\gamma r)}{V(\gamma r)}\right) .
\end{eqnarray*}%
Let $C_{VD}$ denotes the volume doubling constant of $(G,d)$, then%
\[
V(\gamma r)\leq C_{VD}^{1+\log \gamma }V(r),\;\;\frac{V(2^{i}\gamma r)}{%
V(\gamma r)}\leq C_{VD}^{i}. 
\]%
Recall that $m(t)$ is a regularly varying function with 
negative index. Hence, for $\gamma $ large enough, we have
\[
M=\sup_{0<s\leq t,i\in \mathbb N}\left\{\frac{m(s)}{m(t)}\left( \frac{s}{t%
}\right) ^{2^{i-1}\gamma /3}\right\}<\infty 
. 
\]%
Therefore%
$$\mathbf P_e\left( \left\Vert X_{s}^{R}\right\Vert \geq 
\frac{\gamma r}{2}
\right)
\leq C_1 m(t)V(r)e^{-r/6R}\sum_{i=0}^{\infty }e^{-2^{i}\gamma/ 12}C_{VD}^i
$$
By assumption, $r=r(t)$  and $R=R(t)$
satisfy 
$$\sup_{t>0}\left\{m(t)V(r(t))e^{-r(t)/3R(t)}\right\}<\infty 
.$$ 
It follows that for  $\gamma $ sufficiently large, we have 
\[
\mathbf P_e\left( \left\Vert X_{s}^{R}\right\Vert \geq \frac{\gamma }{2}r(t)\right) <%
\frac{\varepsilon }{4}. 
\]%
Plugging this estimate into (\ref{Meyer}), we obtain 
$\mathbf P_e\left( \sup_{s\le t}\{\left\Vert X_{s}\right\Vert\} \geq \gamma r(t)\right) <\varepsilon$. 
\end{proof}

\begin{cor}
\label{cor-tight} Under the hypotheses of \emph{Proposition \ref{tightness}}%
, for any $\epsilon>0$ there exists $\gamma>0$ such that 
\begin{equation*}
\mathbf{P}_e\left( \sup_{s\leq t}\left\Vert X_{s}\right\Vert \leq 2\gamma
r(t),\left\Vert X_{t}\right\Vert \leq \gamma r(t)\right) \geq 1-\varepsilon .
\end{equation*}
\end{cor}

\begin{proof}
Write
\begin{eqnarray*}
\lefteqn{\mathbf P_e\left( \sup_{s\leq t}\left\Vert X_{s}\right\Vert \leq 2\gamma
r(t),\left\Vert X_{t}\right\Vert \leq \gamma r(t)\right)} \hspace{.5in} && \\
&=&\mathbf P_e\left( \left\Vert X_{t}\right\Vert \leq \gamma r(t)\right)
-\mathbf P_e\left( \sup_{s\leq t}\left\Vert X_{s}\right\Vert \geq 2\gamma
r(t),\left\Vert X_{t}\right\Vert \leq \gamma r(t)\right) \\
&\geq &1-2\mathbf P_e\left(\sup_{s\leq t}\left\Vert X_{s}\right\Vert \geq 
\gamma
r(t)\right).
\end{eqnarray*}%
Note that we have used the fact that, 
because of space homogeneity (i.e., group invariance), $X_t$ cannot 
escape to infinity in finite time.
\end{proof}

\begin{rem}
\label{rem-disc} The conclusions of Proposition \ref{tightness} and
Corollary \ref{cor-tight} apply to the associated discrete time random walk.
To see this, fix a regularly varying function $m$ and note that (up to
changing $m$ to $cm$ for some constant $c$), (\ref{odub}) is equivalent to $%
\nu^{(2n)}(e)\le m(n)$. Further, it is easy to control the difference
between $\mathbf{P}_e(\|X_t\|\ge r)$ and $\mathbf{P}_e(\|X_n\|\ge r)$ with $%
n=\lfloor t\rfloor$ as long as $n$ is large enough. It follows that the
proof above applies the discrete random walk result as well.
\end{rem}

\section{Pseudo-Poincar\'e inequality and strong control} \label{sec-psc}
\setcounter{equation}{0}
\subsection{Pseudo-Poincar\'e inequality}

With some work, the results 
of the previous section can be extended to the more general
context of graphs and discrete spaces. The results presented below make a
more significant use of the underlying group structure.

\begin{defin}
Let $G$ be discrete group equipped with a symmetric probability measure $\nu$, 
a sub-additive function $\|\cdot\|$ and a positive continuous increasing
function $r$ with inverse $\rho$. We say that $\nu$ satisfies a pointwise $%
(\|\cdot\|,r)$-pseudo-Poincar\'e inequality if, for any $f$ with finite
support on $G$, 
\begin{equation}  \label{pseudo}
\forall\, g\in G,\;\;\sum_{x\in G}|f(xg)-f(x)|^2\le 
C \rho(\|g\|) \mathcal{E}_\nu%
(f,f)
\end{equation}
Here $\mathcal E_\nu$ is the Dirichlet form of $\nu$ defined at 
(\ref{def-dir}).
\end{defin}

\begin{theo}
\label{theo-PP} Assume that $(G,\|\cdot\|)$ is such that $V$ is doubling.
Let $\nu$ be a symmetric probability measure such that $\nu(e)>0$. Assume
that $r$ is a positive doubling continuous increasing function such that 
\begin{equation*}
\nu^{(2n)}(e)\simeq V(r(n))^{-1}.
\end{equation*}
Assume further that $\nu$ satisfies the $(\|\cdot\|,r)$-pseudo-Poincar\'e
inequality. Then there exists $\eta>0$ such that for all $n$ and $g$ with $%
\|g\|\le \eta r(n)$ we have 
\begin{equation*}
\nu^{(n)}(g)\simeq V(r(n))^{-1}.
\end{equation*}
\end{theo}

\begin{proof}
The hypothesis (\ref{pseudo}) and the argument of \cite[Theorem 4.2]{HSC}
gives
$$|\nu^{(2n+N)}(x)-\nu^{(2n+N)}(e)|\le 
C \left(\frac{\rho(\|x\|)}{N}\right)^{1/2}\nu^{(2n)}(e).$$
Fix $x$ and $n$ such that  $\rho(\|x\|)\le \eta n $ and use the above
inequality with  $N=2n$ to obtain
$$\nu^{(4n)}(x)\ge \left(1-  (2C'\eta)^{1/2}\right) \nu^{(4n)}(e).$$
Hence, we can  choose $\eta>0$ such that
$$\nu^{(4n)}(x)\ge c \nu^{(4n)}(e).$$
Since $\nu(e)>0$, this also holds for $4n+i$, $i=1,2,3$, 
at the cost of changing the value of the positive constant $c$.
\end{proof}

\subsection{Strong control}

\begin{defin}
\label{wellc} We say that $\|\cdot\|$ is well-connected if there exists $%
b\in (0,\infty)$ such that, for any $r>0$ and $x\in G$ with $\|x\|\le r$
there exists a finite sequence of points $(x_0)_1^N \in G$
with $\|x_i\|\le 2r$, $\|x_i^{-1}x_{i+1}\|\le b$, $x_0=e$ and $x_N=x$.
\end{defin}
Note that in this definition, the number $N$ of points in the sequence 
is finite but that it depend on $x$ and no upper bound in terms of $r$ 
is required. 

\begin{lem}
\label{lem-cov} Assume that $\|\cdot\|$ is well-connected and $V$ is
doubling. Then for any fixed $\epsilon>0$ there exists $M_\epsilon$ such
that for any $r\ge 8b /\epsilon$ and any $\|x\|\le r$ we can find $(z_i)_0^M$%
, $z_0=e$, $z_M=x$, $M\le M_\epsilon$, such that $\|z_{i}^{-1}z_{i+1}\|\le
\epsilon r$.
\end{lem}

\begin{proof}  Let $\{y_i: 1\le i\le M'\}$ be a maximal 
$\epsilon  r/4$-separated 
set of points in $B(e, 2r)=\{\|g\|\le 2r\}$. 
The ball $B_i=\{y_i, \epsilon r/9\}$ are disjoints and have volume 
$V(\epsilon r/9)$ comparable to $V(2r)$. Hence $M'\le M'_\epsilon$ for some 
finite $M'_\epsilon$ independent of $r$.   The union of the balls 
$B'_i=\{y_i,\epsilon r/4)$
covers $B(2r)$ (otherwise, $\{y_i: 1\le i\le M'\}$ would not be maximal).
In particular, these balls cover the sequence $(x_i)_0^N$ 
and we can extract a sequence $B^*_i= B'_{j_i}$, $1\le i\le M\le M'_\epsilon$,  
such that  $B^*_1\ni e$, $B^*_M \ni x$ and 
$$\inf\{\|h^{-1}g\|: h\in B^*_i,g\in B^*_{i+1}\}\le b.$$
Set $z_0=e$, $z_i=y_{j_i}$, $1\le i\le N$, $z_{N+1}=z$.
Then $\|z_{i-1}^{-1}z_i\|\le 2 \epsilon r/4 +b \le \epsilon r$ as desired.
\end{proof}

\begin{pro}
\label{pro-SC} Assume that the norm $\|\cdot\|$ is such that $V$ is doubling
and $\|\cdot\|$ is well-connected. Let $r$ be a positive continuous
increasing doubling function. Let $\nu$ be a symmetric probability measure
that is $(\|\cdot\|,r)$-controlled and satisfies $\nu(e)>0$ and a pointwise $%
(\|\cdot\|,r)$-pseudo-Poincar\'e inequality. Then $\nu$ is also strongly $%
(\|\cdot\|,r)$-controlled.
\end{pro}

\begin{proof} First, we show that for any $\kappa>0$ there exists 
$c_\kappa>0$ such that
$\|x\|\le \kappa r(n)$ implies
$$\nu^{(n)}(x) \ge c V(r(n))^{-1}.$$   
By Theorem \ref{theo-PP}, there exists 
$\eta$ such that  $\nu^{(n)}(x) \ge c_1 V(r(n))^{-1}$ for all 
$\|x\|\le \eta r(n)$.   By Lemma \ref{lem-cov}, 
for any fixed $\kappa$ there exists $M_\kappa$ such that for any  $\|x\|\le \kappa r(n)$ we can find $(z_i)_0^M$, $z_0=e$, $z_M=x$, $M\le M_\kappa$, such that  $\|z_{i}^{-1}z_{i+1}\|\le \eta r(n)/4$.
Write $B_i=\{\|z_i^{-1}g\|\le \eta r(n)/4\}$ and 
\begin{eqnarray*}
\nu^{(M n)}(x)&\ge &\sum_{(y_1,\dots, y_{M})\in \otimes_1^M B_i}
\nu^{(n)}(y_1) \cdots  \nu^{(n)}(y^{-1}_iy_{i+1}) \dots \nu^{(n)}(y_{M}^{-1}x)
\\
&\ge & c_1^{M+1} V(\eta r(n)/4)^MV(r(n))^{-M-1}\simeq c'_1 V(r(n))^{-1}.
\end{eqnarray*}
Since $\nu(e)>0$, this shows that
$\|x\|\le \kappa r(n)$ implies
$\nu^{(n)}(x) \ge c V(r(n))^{-1}$ as stated.    
In particular, for any fixed $\kappa$, there exists $\epsilon>0$ such that 
for any $x,n$ with $\kappa r(n)\le \tau$ and $\|x\|\le \tau$,
$$ \mathbf P_x(\|X_n\|\le \tau)\ge \epsilon.$$ 
Now,  fix  $\gamma_1\in (1,\infty)$. Let $\epsilon_0>0$ be such that,
for any $x$, $n$, $\tau$ with   $$\|x\|\le \tau \le 
\gamma_1 r(2n),$$ we have  $\mathbf P_x(\|X_n\|\le \tau)\ge \epsilon_0.$ 
Let $\gamma\ge 1$ be given by Definition \ref{controlled} so that
$$\mathbf P_e\left(\sup_{k\le n}\{\|X_k\|\}\ge \gamma 
r(n)\right)\ge \epsilon_0/2.$$  
Set $\gamma_2=\gamma/\gamma_1+1$ and, for any $x,n,\tau$ with $\|x\|\le \tau$ and
$\frac{1}{2}\rho(\tau/\gamma_1)\le n\le \rho(\tau/\gamma_1)$, 
write
\begin{eqnarray*} 
\lefteqn{\mathbf P_x\left( \sup_{k\leq n}\{\left\Vert X_{k}\right\Vert\} 
\leq \gamma_2 \tau ,\left\Vert X_{n}\right\Vert \leq \tau\right)} \hspace{.5in} && \\
&=&\mathbf P_x\left( \left\Vert X_{t}\right\Vert \leq \tau\right)
-\mathbf P_x\left( \sup_{k\leq n}\left\Vert X_{k}\right\Vert \geq \gamma_2
\tau,\left\Vert X_{n}\right\Vert \leq \tau\right) \\
&\geq & 
\epsilon_0
-\mathbf P_e\left( \sup_{k\leq n}\left\Vert X_{k}\right\Vert \geq \gamma
\tau/\gamma_1 \right) \ge \epsilon_0/2.
\end{eqnarray*}%
This proves that $\nu$ is  strongly
$(\|\cdot\|,r)$-controlled.  \end{proof} 

As a simple illustration of these techniques, consider the case of of an
arbitrary symmetric measure $\nu$ with generating support and finite second
moment (with respect to the word-length $|\cdot|$) on a group with
polynomial volume growth of degree $D(G)$. It follows from \cite{PSCstab}
that $\nu^{(n)}(e)\simeq n^{-D(G)/2}$ and satisfies a pointwise classical
pseudo-Poincar\'e inequality (with $\rho(t)=t^2$). Proposition \ref{pro-SC} yields the following
result.

\begin{theo}
\label{th-eq5} Let $G$ be a finitely generated group with polynomial volume
growth with word-length $|\cdot|$. Assume that $\nu$ is symmetric, satisfies 
$\mu(e)>0$, has generating support and satisfies $\sum|g|^2\nu(g)<\infty$.
Then $\nu$ is strongly $(|\cdot|, t\mapsto t^{1/2})$-controlled.
\end{theo}

\section{Measures supported on powers of generators}\label{sec-pow}

\setcounter{equation}{0}

\subsection{The measure $\protect\mu_{S,a}$}

In this subsection we consider the special case when $G$ is a nilpotent
group equipped with a generating $k$-tuple $S=(s_1,\dots,s_k)$ and 
\begin{equation}  \label{JmuSa}
J(x,y)=\mu_{S,a}(x^{-1}y),\;\;a=(\alpha_1,\dots,\alpha_k)\in (0,\infty)^k
\end{equation}
with $\mu_{S,a}$ given by (\ref{muSa}). Our aim is to prove the first 
statement in Theorem \ref{th-muSa}. The study of the random walks driven
by this class of measure was initiated by the authors in \cite{SCZ-nil} and
we will refer to and use some of the main results of \cite{SCZ-nil}.

Following \cite[Definition 1.3]{SCZ-nil}, let $\mathfrak{w}$ be the power
weight system on the formal commutators on the alphabet $S$ associated with
setting $w_i=1/\widetilde{\alpha}_i$, $\widetilde{\alpha}_i=\min\{\alpha_i,2%
\}$. Namely, The weight of any commutator $c$ using the sequence of letters $%
(s_{i_1},\dots,s_{i_m})$ from $S$ (or their formal inverse) is $%
w(c)=\sum_1^m w_{i_j}$. In \cite{SCZ-nil}, the authors proved the following
result.

\begin{theo}[\protect\cite{SCZ-nil}]
\label{th-stab} Referring to the above setting and notation, assume that the
subgroup of $G$ generated by $\{ s_i: \alpha_i<2\}$ is of finite index. Then
there exists a real $D_{S,a}=D(S,\mathfrak{w})$ such that 
\begin{equation*}
Q_{S,a}(r) \simeq r^{D_{S,a}},\;\; \mu_{S,a}^{(n)}(e)\simeq n^{-D_{S,a}}.
\end{equation*}
The real $D_{S,a}=D(S,\mathfrak{w})$ is given by \emph{\cite[Definition 1.7]%
{SCZ-nil}}. Further, there exists a $k$-tuple $b=(\beta_1,\cdots,\beta_k)\in
(0,2)^k$ such that $\beta_i=\alpha_i$ is $\alpha_i<2$, $D(S,a)=D(S,b)$, and 
\begin{equation*}
\forall\, g\in G, \|g\|_{S,a}^{\alpha_*}\simeq \|g\|_{S,\beta}^{\beta_*}.
\end{equation*}
In addition, $\mu_{S,a}$ satisfies a pointwise $(\|\cdot\|_{S,a},t\mapsto
t^{1/\alpha_*})$-pseudo-Poincar\'e inequality.
\end{theo}

The volume estimate $Q_{S,a}(r) \simeq Q_{S,b}(r)\simeq r^{D_{S,a}}$ shows,
in particular, that $(G,\|\cdot\|_{S,b})$ has the volume doubling property.
The upper bound $\mu_{S,a}^{(n)}(e)\le C n^{-D_{S,a}}$ implies that the
continuous time process with jump kernel $J$ defined above satisfies 
\begin{equation*}
\forall\,t>0, x\in G,\;\;p(t,x,x)\le m(t)=Ct^{-D_{S,a}}.
\end{equation*}

Note that $\|\cdot\|$ is clearly well-connected (Definition \ref{wellc}). In
order to apply Propositions \ref{tightness} and \ref{pro-SC} to the present
case and prove Theorem \ref{th-muSa}, it clearly suffices to prove the
following lemma which provides estimates for $\delta_R$ and $\mathcal{G}(R)$.

\begin{lem}\label{lem-stab}
Referring to the setting and hypotheses of \emph{Theorem \ref{th-stab}}, for 
$J$ given by \emph{(\ref{JmuSa})}, let $\|\cdot\|=\|\cdot\|_{S,b}$, $%
D=D_{S,b}=D_{S,a}$, we have 
\begin{eqnarray*}
V(r) &=&\#\{g\in G: \|g\|\le r\} \simeq r^{D\beta _*}, \\
\delta _{R} &\simeq &R^{-\beta _{*}}, \\
\mathcal{G}(R) &\simeq &R^{2-\beta_*}.
\end{eqnarray*}
\end{lem}

\begin{proof} The volume estimate follows immediately from Theorem \ref{th-stab}.
Let $\mathfrak v$ be the power weight system associated with $b$
(in particular, $v_i=1/\beta_i >1/2$). 
By \cite[Proposition 2.17]{SCZ-nil},  for each $i$ there exists 
$0<\beta'_i\le\beta_i\le \beta_*<2$ such that
\[ 
\left\Vert s_{i}^{n}\right\Vert \simeq |n|^{\beta'_i/\beta_*}. 
\]%
In the notation of \cite{SCZ-nil}, 
$\beta_i= \overline{v}_{j_{\mathfrak v}(s_i)}$.We have
\begin{eqnarray*}
\bigskip \delta _{R} &=&\sum_{\left\Vert x\right\Vert >R}\mu
_{S,a}(x) 
=\sum_{i=1}^{k}\sum_{\left\Vert s_{i}^{n}\right\Vert>R}\frac{\kappa_i}{(1+|n|)^{1+\alpha_i}}\\
&\simeq &\sum_{i=1}^{k}\sum_{n >R^{\beta_*/\beta'_i}}
\frac{\kappa_i}{(1+|n|)^{1+\alpha_i}}
\simeq \sum_{i=1}^{k} R^{-\beta_*\alpha_i/\beta'_i}
\simeq R^{-\beta_*}.
\end{eqnarray*}%
The last estimate use that fact that there must be some  
$i\in \{1,\dots, k\}$ such that
$\alpha_i=\beta'_i$ and that, always, $\alpha_i\ge \beta'_i$.

Similarly, since $\alpha_i\ge \beta'_i$ and $\beta_*<2$, we have
$2\beta'_i/\beta_*-\alpha_i>0$. This yields
\begin{eqnarray*}
\mathcal{G}(R) &=&\sum_{\left\Vert x\right\Vert \le R}\left\Vert
x\right\Vert ^{2}\mu _{S,a}(x) 
=\sum_{i=1}^{k}\sum_{\left\Vert s_{i}^{n}\right\Vert \leq
R}\frac{\kappa_i \left\Vert s_{i}^{n}\right\Vert^{2}}{(1+|n|)^{1+\alpha_i}} \\
&\simeq&\sum_{i=1}^{k}\sum_{0\le n \leq
R^{\beta_*/\beta'_i}}\frac{\kappa_i |n|^{2\beta'_i/\beta_*}}{(1+|n|)^{1+\alpha_i}}
\simeq \sum_{i=1}^{k}R^{2-\beta_*\alpha_i/\beta'_i} 
\simeq R^{2-\beta_*}.
\end{eqnarray*}
This proves Lemma \ref{lem-stab}.
\end{proof}

\subsection{Some regular variation variants of $\protect\mu_{S,a}$}

Consider the class of measure $\mu$ of the form 
\begin{equation}  \label{slow-li}
\mu(g)= \frac{1}{k}\sum_1^k\sum_{m\in \mathbb{Z}} \frac{\kappa_i \ell_i(|n|)%
}{(1+|n|)^{1+\alpha_i}}
\end{equation}
where each $\ell_i$ is a positive slowly varying function satisfying $%
\ell_i(t^b)\simeq \ell_i(t)$ for all $b>0$ and $\alpha_i\in (0,2)$. For each 
$i$, let $F_i$ be the inverse function of $r\mapsto r^{\alpha_i}/\ell_i(r)$.
Note that $F_i$ is regularly varying of order $1/\alpha_i$ and that $%
F_i(r)\simeq [r\ell_i(r)]^{1/\alpha_i}$, $r\ge 1$, $i=1,\dots,k$. We make
the fundamental assumption that the functions $F_i$ have the property that
for any $1\le i,j\le k$, either $F_i(r)\le C F_j(r)$ of $F_j(r)\le CF_i(r)$.
For instance, this is clearly the case if all $\alpha_i$ are distinct. Set $%
a=(\alpha_1,\dots,\alpha_k)\in (0,2)^k$ and consider also the power weight
system $\mathfrak{v}$ generated by $v_i=1/\alpha_i$, $1\le i\le k$, as in 
\cite[Definition 1.3]{SCZ-nil}. Fix $\alpha_0\in (0,2)$ such that 
\begin{equation*}
\alpha_0>\max\{\alpha_i: 1\le i\le k\}
\end{equation*}
and $\alpha_0/\alpha_i \not\in \mathbb{N}$, $i=1,\dots, k$. Observe that
there are convex functions $K_i\ge 0$, $i=0,\dots,k$, such that $K_i(0)=0$
and 
\begin{equation}  \label{conv}
\forall r\ge 1,\;\; F_i(r^{\alpha_0}) \simeq K_i(r).
\end{equation}
Indeed, $r\mapsto F_i(r^{\alpha_0})$ is regularly varying of index $%
\alpha_0/\alpha_i$ with $1<\alpha_0/\alpha_i\not\in \mathbb{N}$. By \cite[%
Theorems 1.8.2-1.8.3]{BGT} there are smooth positive convex functions $%
\tilde{K}_i$ such that $\tilde{K}_i(r)\sim F_i(r^{\alpha_0})$. If $\tilde{K}%
_i(0)>0$, it is easy to construct a convex function $K_i:[0,\infty)%
\rightarrow [0,\infty)$ such that $K_i\simeq \widetilde{K}_i$ on $[1,\infty)$
and $K_i(0)=0$. Let use $\mathfrak{K}$ to denote the collection $(K_i)_1^r$.

Now, set 
\begin{equation*}
\|g\|=\|g\|_{\mathfrak{K}} = \min\left\{ r: g= \prod _{j=1}^m
s_{i_j}^{\epsilon_j}: \epsilon_j=\pm 1,\;\; \#\{j: i_j =i\}\le
K_i(r)\right\}.
\end{equation*}
Because of the convexity property of the $K_i$, $\|\cdot\|$ is a norm. Note
also that it is well-connected. The following theorem is proved in \cite%
{SCZ-nil}.

\begin{theo}
Referring to the above notation and hypothesis, there exists a real $%
D=D_{S,a}=D(S,\mathfrak{v})$ and a positive slowly varying function $L$
(explicitly given in \emph{\cite[Theorem 5.15]{SCZ-nil}} and which satisfies
$L(t^a)\simeq L(t)$ for all $a>0$) such that:

\begin{itemize}
\item For all $r\ge 1$, $V(r)=\#\{g:\|g\|\le r\}\simeq r^{\alpha_0D}L(r)$

\item For a each $1\le i\le k$, there exists a regularly varying
function $\widetilde{F}_i$ such that $%
\|s_i^n\|^{\alpha_0}\leq C \widetilde{F}^{-1}_i(n)$ where $\widetilde{F}%
_i\ge F_i$ and with equality for some $1\le i \le r$.

\item For all $n\ge 1$, $\mu^{(2n)}(e)\le C(n^{D} L(n))^{-1}.$

\item The measure $\mu$ satisfies a pointwise $(\|\cdot\|,t\mapsto
t^{1/\alpha_0})$-pseudo-Poincar\'e inequality.
\end{itemize}
\end{theo}

Here, we prove the following result.

\begin{theo}
\label{th-eq3} Let $G$ be a finitely generated nilpotent group equipped with
a generating $k$-tuple $(s_1,\dots,s_k)$. Assume that $\mu$ is a probability
measure on $G$ of the form \emph{(\ref{slow-li})}. Let $\ell_i$, $F_i$, $L$, 
$D=D_{S,a}$, $\alpha_0\in (0,2)$ and $\|\cdot\|$ be as described above. Then 
$\mu$ is strongly $(\|\cdot\|, t\mapsto t^{1/\alpha_0})$-controlled.
\end{theo}

\begin{proof} It suffices to estimate the quantities $\delta_R$ and $\mathcal G(R)$
in the present context.
For $\delta_R$, we
have
$$\delta_R \simeq \sum _1^k \sum_{n\ge \widetilde{F}_i(R^{\alpha_0})}\frac{1}{nF^{-1}_i(n)}
\simeq  \sum_1^k \frac{1}{F_i^{-1}\circ \widetilde{F}_i(R^{\alpha_0})}
\simeq R^{-\alpha_0}. $$  
A similar computation gives
$$\mathcal G(R)\simeq \sum_1^k 
\frac{ R^2}{F_i^{-1}\circ \widetilde{F}_i(R^{\alpha_0})}\simeq R^{2-\alpha_0}.$$
\end{proof}

\subsection{The critical case when $\protect\alpha_i=2$, $1\le i\le k$}

When $a=\mathbf{2}=(2,\dots, 2)$, that is, $\alpha_i=2$ for all $1\le i\le k$%
, we work with the usual word-length function $|g|$ associated with the
generating set $\mathcal{S}=\{s_1^{\pm 1},\dots, s_k^{\pm 1}\}$. In this
case, $V(r)=\#\{g:|g|\le r\}\simeq r^{D(G)}$ where $D(G)$ is the classical
degree of polynomial growth for the nilpotent group $G$. It is proved in 
\cite{SCZ-nil} that $\mu_{S,\mathbf{2}}^{(n)}(e)\le C (n\log n)^{-D/2}$ and
that $\mu_{S,\mathbf{2}}$ satisfies a pointwise $(|\cdot|,t\mapsto (t\log
t)^{1/2})$-pseudo-Poincar\'e inequality. Further, $|s_i^n|\simeq
|n|^{1/\beta_i}$ with $\beta_i\ge 1$ and $\beta_i=1$ for some $i$. From this
it easily follows that 
\begin{equation*}
\delta_R\simeq R^{-2},\;\;\mathcal{G}(R)\simeq \log R .
\end{equation*}
Applying Proposition \ref{pro-SC} with $r(t)=(t\log t)^{1/2}$ yields the
following theorem.

\begin{theo}
\label{th-eq4} Let $G$ be a finitely generated nilpotent group equipped with
a generating $k$-tuple $(s_1,\dots,s_k)$. Let $D(G)$ be the volume growth
degree of $G$. Then $\mu_{S,\mathbf{2}}$ is strongly $(|\cdot|,t\mapsto
(t\log t)^{1/2})$-controlled.
\end{theo}

\section{Norm-radial measures} \label{sec-rad}

\setcounter{equation}{0}

In this section we assume that $G$ is a finitely generated group with
polynomial volume growth of degree $D(G)$ and we consider norm-radial
symmetric probability measures.

\subsection{Radial measures with stable-like exponent $\alpha\in (0,2)$}
This subsection treats probability measures of the form 
\begin{equation}  \label{Jrad}
\nu_\alpha (x) \simeq \frac{1}{(1+\left\Vert x\right\Vert) ^{\alpha }
V(\|x\| )}, \;\; J(x,y)\simeq \nu_{\alpha}(x^{-1}y),
\end{equation}
where $\alpha \in (0,2]$, $\left\Vert \cdot \right\Vert $ is a norm on $G$
and $V(r)=\#\{g:\|g\|\le r\}$.  
The case when $\alpha\in (0,2)$ and 
$V(r)\simeq r^d$ for some $d$ is treated
in \cite{BBK,BGK,MSC} where global matching upper and lower bounds are obtained.
We note that \cite{BBK,BGK,MSC} are set in  more general contexts where the group
structure play no role.   
We start with the following easy observation.
\begin{lem}
\label{DGrad} Referring the situation described above, assume that $V$
satisfies $V(2r)\le C_{DV} V(r)$ for all $r>0$. Then $\delta_R \simeq
R^{-\alpha}$ and 
\begin{equation*}
\mathcal{G}(R)\simeq \left\{%
\begin{array}{ll}
R^{2-\alpha} & \mbox{ if } \alpha\in (0,2) \\ 
\log R & \mbox{ if } \alpha=2.%
\end{array}%
\right.
\end{equation*}
\end{lem}

\begin{proof} This follows by inspection.
\end{proof}

The next lemma follows by application of Proposition \ref{tightness}.
However, in this lemma, we make a significant 
hypothesis on $\nu_{\alpha}^{(n)}(e)$.
\begin{lem}
\label{lem-strange} Set $r_\alpha(t)= t^{1/\alpha}$ if $\alpha\in (0,2)$, $%
r_2(t)=(t\log t)^{1/2}$. Referring the situation described above, assume
that $V$ is regularly varying of positive index and 
\begin{equation}\label{hyp-alpha}
\nu_\alpha^{(n)}(e)\le C V(r_\alpha(n))^{-1}.
\end{equation}
Then $\nu_\alpha$ is $(\|\cdot\|,r_\alpha)$-controlled.
\end{lem}

The next theorem provides a basic class of  examples when the hypothesis 
(\ref{hyp-alpha}) regarding $\nu_\alpha$ can indeed be verified. 
Note that the result is restricted to the case $\alpha\in (0,2)$.

\begin{theo}
\label{th-Mat} Referring the situation described above, assume that $V$ is
regularly varying of positive index and $\alpha\in (0,2)$. Then 
\begin{equation*}
\nu_\alpha^{(n)}(e)\simeq V(n^{1/\alpha})^{-1}
\end{equation*}
and $\nu_\alpha$ is $(\|\cdot\|,t\mapsto t^{1/\alpha})$ controlled.
\end{theo}

\begin{proof}It suffices to prove the upper bound 
$\nu_\alpha^{(n)}(e)\le CV(n^{1/\alpha})^{-1}$.  Start by checking that
$$\nu_{\alpha}(x)\simeq \sum_0^\infty \frac{1}{(1+m)^{1+\alpha}}
\frac{\mathbf1 _{B(m)}(x)}{V(m)}.
$$ 
where $B(m)=\{x\in G: \|x\|\le m\}$. Then apply the elementary technique 
of \cite[Section 4.2]{BSClmrw} to derive the desired upper 
bound on $\nu_\alpha^{(n)}(e)$.
\end{proof} 

\begin{rem}
In the context of Theorem \ref{th-Mat}, we do not know if $\|\cdot\|$ is
well-connected and we also do not know if $\nu_\alpha$ satisfies a \emph{%
pointwise} $(\|\cdot\|,r_\alpha)$-pseudo-Poincar\'e inequality. Hence, the
techniques used in this paper do not suffice to obtain strong control.
However, if $\|\cdot\|$ is well-connected and $\nu_\alpha$ satisfies a
pointwise $(\|\cdot\|,r_\alpha)$-pseudo-Poincar\'e inequality then the
strong $(\|\cdot\|,r_\alpha)$-control follows by Proposition \ref{pro-SC}.
This proves the second statement in Theorem \ref{th-muSa}.
\end{rem}

\subsection{Word-length radial measures}

As noticed above, the study of $\nu_\alpha$ in the  case $\alpha=2$ 
is significantly more difficult than in the case $\alpha\in (0,2)$.
In fact, we do not
know how to treat this case in the generality described in the previous 
subsection. The following
theorem treats the case when $\|\cdot\|$ is the usual word-length function $%
\|\cdot\|=|\cdot|$ on $G$.

\begin{theo}\label{th-alpha2}
Assume that $G$ is a group of polynomial volume growth equipped with
generating $k$-tuple $S=(s_1,\dots,s_k)$ 
and the associated word-length $|\cdot|$ and volume function $V$. Let $D(G)$ be
the degree of polynomial volume growth of $G$. Let $\nu_2$ be a symmetric 
probability measure such that 
$$\nu_2(g)\simeq ((1+|g|)^2V(|g|))^{-1}.$$
Then the have
\begin{equation*}
\nu_2^{(n)}(e)\simeq (n\log n )^{-D(G)/2}.
\end{equation*}
Further, $\nu_2$ is strongly $(|\cdot|,t\mapsto (t\log t)^{1/2})$-controlled.
\end{theo}

\begin{proof} We apply Lemma \ref{lem-strange} and Proposition \ref{pro-SC}.
When $G$ is nilpotent, 
the upper bound $\nu_2^{(n)}(e)\le (n\log n)^{-D(G)/2}$ 
follows from Theorems 4.8 and 5.7 of \cite{SCZ-nil}. 
Namely, \cite[Theorem 5.7]{SCZ-nil}
shows that 
$$\nu_2^{(n)}(e) \le C\mu_{S,\mathbf 2}^{(Kn)}(e)$$ and 
\cite[Theorem 4.8]{SCZ-nil} gives 
$\mu_{S,\mathbf 2}^{(n)}(e)\le C(n\log n)^{-D(G)/2}$. 
Further, \cite[Theorem 5.7]{SCZ-nil} shows that $\nu_2$ satisfies 
a pointwise $(|\cdot|,t\mapsto (t\log t)^{1/2})$-pseudo-Poincar\'e inequality.

Since any group of polynomial volume growth of degree $D(G)$ contains a
nilpotent subgroup of finite index 
(hence, with the same degree of polynomial volume growth)
the upper bound $\nu_2^{(n)}(e)\le C(n\log n )^{-D(G)/2}$
follows from the comparison theorem \cite[Theorem 2.3]{PSCstab}.
By direct inspection, the desired pseudo-Poincar\'e inequality also follows.
\end{proof}

Note that Theorem \ref{th-phi} includes the result stated in Theorem 
\ref{th-alpha2} as a special case and  provides a very satisfactory result covering the 
behaviors of word-length radial measures across the  second moment threshold.

\begin{proof}[Proof of Theorem \ref{th-phi}] 
The same technique of proof as for Theorem \ref{th-alpha2}
gives the much more complete and subtle result
stated in the introduction as Theorem \ref{th-phi}. Namely, let $\phi:[0,\infty)\ra [1,\infty)$ be
a continuous increasing regularly function of index $2$ and  
let $\nu_\phi$ be as in (\ref{nuphi}), that is, assume that 
$\nu_\phi$ is symmetric and satisfies $\nu(g)\simeq [\phi(|g|)V(|g|)]^{-1}$.  
First, assume that $G$ is nilpotent and 
let $\mu_{S,\phi}$ be the measure given by
$$\mu_{S,\phi}(g)=\frac{1}{k}\sum_1^k\frac{\kappa}{(1+|n|)\phi(n)}\mathbf 1_{s_i^n}(g).$$  Let $r$ be the inverse function of 
$t\mapsto t^2/\int_0^t\frac{sds}{\phi(s)}$.  By \cite[Lemma 4.4]{SCZ-nil},
the measure $\mu_{S,\phi}$ satisfies the pointwise 
$(|\cdot|,r)$-pseudo-Poincar\'e 
inequality. By \cite[Theorem 4.1]{SCZ-nil}, it follows that
$\mu_{S,\phi}^{(n)}(e)\le V(r(n))^{-1}$. 
By \cite[Theorem 5.7]{SCZ-nil}, we have the Dirichlet form comparison
$\mathcal E_{\mu_{S,\phi}}\le C\mathcal E_{\nu_\phi}$. 

Now, if $G$ has polynomial volume growth then it contains  
a  nilpotent subgroup with finite index, $G_0$. By inspection, quasi-isometry 
and comparison of Dirichlet forms (see \cite{PSCstab}),  
it is easy to transfer both the pointwise
 $(|\cdot|,r)$-pseudo-Poincar\'e inequality
and the decay $\nu^{(n)}_\phi(e)\le V(r(n))^{-1}$ 
from $G_0$ to  $G$.    Further, one checks that the functions $\delta_R$
and $\mathcal G(R)$ satisfy  $\delta_R \simeq 1/\phi(R)$ and 
$\mathcal G(R)\simeq  \int_0^R\frac{tdt}{\phi(t)}$.  Proposition
\ref{tightness} with  $r(t)=R(t)$ equals to the inverse function of 
$s\mapsto s^2/ \int_0^s\frac{tdt}{\phi(t)}$ shows that $\nu_\phi$ is
$(|\cdot|,r)$-controlled.  By Proposition \ref{pro-SC}, 
$\nu_\phi$ is strongly
$(|\cdot|,r)$-controlled.
\end{proof}
 
\subsection{Assorted further applications}
The approach presented here is applicable even in cases were we are not able
to obtain sharp results and we illustrate this by an example. Le $G$ be a
nilpotent group equipped with a generating $k$-tuple $S=(s_1,\dots,s_k)$.
Fix $a\in (0,2]^k$ and set $\alpha_*=\max\{\alpha_i, 1\le i\le k\}\le 2$.
Consider the norm $\|\cdot\|_{S,a}$ defined at (\ref{normSa}). Let $\nu_*$
be any symmetric probability measure such that 
\begin{equation*}
\nu_*(g)\simeq \frac{1}{(1+\|g\|_{S,a})^2V(\|g\|_{S,a})},\;\;
V(r)=\#\{g:\|g\|_{S,a}\le r\}.
\end{equation*}
Theorems 3.2, 4.8 and 5.7 of \cite{SCZ-nil} gives the following information.
There exists two reals $D=D_{S,a}$ and $d=d_{S,a}$ and a constant $C_1\in
(0,\infty)$ such that 
\begin{eqnarray}
\nu_*^{(n)}(e) &\le & C_1 n^{-\alpha_* D /2} (\log n)^{-d}  \label{*1} \\
V(r)&\simeq & r^{\alpha_* D} .  \label{*2}
\end{eqnarray}

\begin{theo}
For the probability measure $\nu_*$ on a finitely generated nilpotent group
as described above, we have 
\begin{equation*}
c(\log \log n)^{-\alpha_* D} (n\log n )^{-\alpha_* D /2}\le
\nu_*^{(n)}(e)\le C n^{-\alpha_* D/2} (\log n)^{-d}.
\end{equation*}
\end{theo}

\begin{proof}
The volume estimate (\ref{*2}) and Lemma \ref{DGrad}  gives 
$\delta_R\simeq R^{-2}$ and $\mathcal G(R)\simeq \log R$. 
In order to apply Proposition \ref{tightness}, 
we set $R(t)\simeq (t\log t)^{1/2}$.  Further, we use (\ref{*1})--(\ref{*2})
to verify that the choice  $r(t)= 6A R(t) \log\log t$ with $A$ large enough 
satisfies the condition of Proposition \ref{tightness}. Indeed, we have 
$m(t)\simeq t^{-\alpha_* D/2}(\log t)^{-d}$, $V(r)\simeq r^{\alpha_* D}$ so that
$$m(t) V(r(t))e^{-r(t)/6R(t)}  \le C
(\log t)^{-d + \alpha_* D/2} (\log\log t)^{\alpha_* D}e^{-A \log \log t}.$$ 
Clearly, for $A$ large enough, the right-hand side is bounded above by 
a constant as required by Proposition \ref{tightness} which now gives the 
stated lower bound on $\nu_*^{(n)}(e)$. 
\end{proof}

\subsection{Complementary off-diagonal upper bounds}
In contrast with the case (\ref{muSa}) of measures supported on powers of
generators, for norm-radial kernels of type (\ref{Jrad}), we can use Meyer's
construction to derive good off-diagonal bounds for $p(t,e,x)$.

\begin{pro}
\label{up1} Let $G$ be a finitely generated group equipped with a norm $%
\|\cdot \|$. For $\alpha\in (0,2)$, let $\nu_\alpha$ be a symmetric
probability measure on $G$ satisfying \emph{(\ref{Jrad})}. Assume that there
exist a positive slowly varying function $\ell$ and a real $D>0$ such that:
\begin{enumerate}
\item $\forall\, r>1,\;\; V(r)\simeq r^D\ell(r)$;
\item $\forall\,t>0,\;x\in G,\;\;p(t,x,x)\le m(t)\simeq [(1+t)^{D/\alpha} \ell(t^{1/\alpha})]^{-1}$.
\end{enumerate}
Then there exists $C$ such that, for all $t>1$ and $x\in G$, we have
\begin{equation*}
p(t,e,x)\leq Cm(t)\min \left\{\left( \frac{t}{\| x\|^\alpha} \right) ^{1+D
/\alpha }\frac{\ell_1(t^{1/\alpha})}{\ell_1(\|x\|)},1\right\} .
\end{equation*}
\end{pro}
\begin{rem}This proposition is stated here mostly for comparison with the 
next proposition. In fact, for the measure $\nu_\alpha$ with $\alpha\in (0,2)$,
the hypothesis (1) implies automatically that (2) is satisfied as well.
See \cite{BBK,BGK,HK,MSC}. See \cite{MSC} for a complete study of this case including two-sided discrete time estimates. 
\end{rem}

\begin{pro}
\label{up2} Let $G$ be a finitely generated group equipped with a norm $%
\|\cdot \|$ with volume $V$. Let $\nu_2$ be a symmetric probability measure
on $G$ satisfying \emph{(\ref{Jrad})} with $\alpha=2$. Assume that:
\begin{enumerate}
\item $\forall\, r>1,  V(r)\simeq r^D$,
\item $\forall t>1,\;x\in G,\;\;p(t,x,x)\le m(t)\simeq (t\log t)^{-D/2}$.
\end{enumerate}
Then there exists $C$ such that, for all $t>1$ and $x\in G$, we have
\begin{equation*}
p(t,e,x)\leq Cm(t)\min \left\{\left( \frac{t \log \|x\|}{\| x\|^2 }\right)
^{1+D /2},1\right\}
\end{equation*}
Further, for any $\gamma\in (0,2)$, there exist $C_\gamma$ such that if 
$1\le t\le \|x\|^\gamma$ then 
\begin{equation*}
p(t,e,x)\le \frac{C_\gamma}{t^{D/2}}\left( \frac{t}{\|x\|^2}\right)^{1+D/2}.
\end{equation*}
\end{pro}

\begin{proof}[Proof of Proposition \ref{up1}]
Under the stated hypothesis, we have $\delta_R\simeq R^{-\alpha}$ and  $\mathcal G(R)\simeq R^{2-\alpha}$ and, for $1\le t\le \eta R^{\alpha}$ (with $\eta$ 
to be fixed later, small enough),   
Proposition\ref{Truncated} gives
\[p_{R}(t,e,x)\leq Cm(t)\left( \frac{t}{\| x\|^{\alpha}}\right) ^{\|x\|/3R}.\] 
By Meyer's construction, we have 
\begin{eqnarray*}
p(t,x,y)&\leq & p_{R}(t,x,y)+t\left\Vert \nu _{R}^{\prime }\right\Vert _{\infty
}\\
&\le& C_1\frac{1}{t^{D/\alpha}\ell_1(t^{1/\alpha})}
\left( \frac{t}{R^{\alpha}}\right) ^{\|x\|/3R}
+\frac{t}{R^{\alpha(1+D/\alpha)} \ell_1(R) }.
\end{eqnarray*}
Choose $R=R(x,t)$ such that the two terms of the sum on the left-hand side are 
essentially equal, namely, set
$$\left(\log \frac{R^\alpha}{t}\right) 
\frac{\|x\|}{3R} = \left(\log \frac{R^\alpha}{t}\right)\left(1+\frac{D}{\alpha}\right) +\log \frac{\ell_1(R)}{\ell_1(t^{1/\alpha})}.$$ 
As long as $\eta$ is small enough, this choice of $R$ gives
$\|x\|\simeq  R$ and
$$
p(t,x,y)\leq \frac{2t}{\| x\|^{\alpha(1+D/\alpha)} \ell_1(\|x\|) }
\simeq  
\frac{1}{t^{D/\alpha}\ell_1(t^{1/\alpha})} 
\left( \frac{t}{\| x\|^{\alpha}}\right) ^{1+D/\alpha}\frac{\ell_1(t^{1/\alpha})}{\ell_1(\|x\|)}.$$
For any $t$ (in particular, $t\geq \eta R^\alpha$) 
we can also use $m(t)$ for an easy upper bound. 
This gives
$$
p(t,e,x) \leq  C m(t)\min 
\left\{ \left( \frac{t}{\left\Vert x\right\Vert^\alpha
}\right) ^{1+D/\alpha }
\frac{\ell_1(t^{1/\alpha})}{\ell_1(\|x\|)}
,1\right\}$$
or, equivalently,
$$p(t,e,x) \leq  C \min 
\left\{t\nu_\alpha(\|x\|)
,m(t)\right\}.$$
\end{proof}

\begin{proof}[Proof of Proposition \ref{up2}]
In the context of proposition \ref{up2}, we have $\delta_R\simeq R^2$
and $\mathcal G(R)\simeq \log R$.  For $1\le t\le \eta R^2$, $\eta>0$ small enough,
Proposition \ref{Truncated} and Meyer's decomposition gives
\begin{eqnarray*}
p(t,x,y)&\leq & p_{R}(t,x,y)+t\left\Vert \nu _{R}^{\prime }\right\Vert _{\infty
}\\
&\le& C(t\log t)^{-D/2}\left( \frac{t \log R}{R^{2}}\right) ^{\|x\|/3R}
+ \frac{t}{R^{2+D}}.
\end{eqnarray*}
If $ R^2/\log R \le t\le R^2$ then this  bound is not better than the easy
bound $p(t,x,y)\le m(t)$. By taking $R$ such that $\|x\|=3R(1+D/2)$, we obtain
$$
p(t,x,y)\leq C m(t)\min
\left\{ \left(\frac{t\log \|x\|}{\|x\|^2}\right)^{1+D/2},1\right\}.$$
However, if $1\le t\le \eta \|x\|^\gamma$ with $\gamma\in (0,2)$ and 
$\eta$ small enough, then we can 
choose $R \simeq \|x\|$ so that
$$(t\log t)^{-D/2}\left( \frac{t \log R}{R^{2}}\right) ^{\|x\|/3R}
= \frac{t}{R^{2+D}},$$
equivalently,
$$\left( \frac{R^2}{t \log R}\right) ^{\|x\|/3R}
= \left(\frac{R^{2}}{t\log R}\right)^{1+D/2} \frac{(\log R)^{1+D/2}}
{(\log t)^{D/2}}. $$
In the region $t\le \|x\|^\gamma$, this yields
$$p(t,e,x)\le  \frac{2t}{\|x\|^{2+D}}\simeq  t^{-D/2}
\left(\frac{ t}{\|x\|^2}\right)^{1+D/2} .$$
\end{proof}

\section{Random walks on wreath products} \label{sec-wp}

In this subsection, we illustrate how to use Proposition \ref{pro-StC} to
derive a lower bound for return probability of certain classes of random
walks on wreath products.

First we briefly review the definition of wreath products and a special type of
random walks on them. Our notation follows \cite{Pittet2002}. Let $H$, $K$
be two finitely generated groups. Denote the identity element of \ $K$ by $%
e_K$ and identity element of $H$ by $e_H$ Let $K_{H}$ denote the direct sum:%
\begin{equation*}
K_{H}=\sum_{h\in H}K_{h}.
\end{equation*}%
The elements of $K_{H}$ are functions $f:H\rightarrow K$, $h\mapsto f(h)=k_h$%
, which have finite support in the sense that $\{h\in H: f(h)=k_h\neq e_K\}$
is finite. Multiplication on $K_H$ is simply coordinate-wise multiplication.
The identity element of $K_H$ is the constant function $\boldsymbol{e}%
_K:h\mapsto e_K$ which, abusing notation, we denote by $e_K$. The group $H$
acts on $K_{H}$ by translation:%
\begin{equation*}
\tau _{h}f(h')=f(h^{-1}h^{\prime }),\;\;h,h^{\prime }\in H.
\end{equation*}%
The wreath product $K\wr H$ is defined to be semidirect product%
\begin{equation*}
K\wr H=K_{H}\rtimes _{\tau }H,
\end{equation*}%
\begin{equation*}
(f,h)(f^{\prime },h^{\prime })=(f \tau _{h}f^{\prime },hh^{\prime }).
\end{equation*}%
In the lamplighter interpretation of wreath products, $H$ corresponds to the
base on which the lamplighter lives and $K$ corresponds to the lamp. We
embed $K$ and $H$ naturally in $K\wr H$ via the injective homomorphisms%
\begin{eqnarray*}
k &\longmapsto &\underline{k}=(\boldsymbol{k}_{e_H},e_H), \;\;\boldsymbol{k}%
_{e_H}(e_H)= k,\; \boldsymbol{k}_{e_H}(h)=e_K \mbox{ if } h\neq e_H \\
h &\longmapsto &\underline{h}=(\boldsymbol{e}_K,h).
\end{eqnarray*}%
Let $\mu $ and $\eta $ be probability measures on $H$ and $K$ respectively.
Through the embedding, $\mu $ and $\eta $ can be viewed as probability
measures on $K\wr H.$ Consider the measure 
\begin{equation*}
q=\eta \ast \mu \ast \eta
\end{equation*}%
on $K\wr H$. This is called the switch-walk-switch measure on $K\wr H$ with
switch-measure $\eta$ and walk-measure $\mu$.

Let $(X_{i})$ be the random walk on $H$ driven by $\mu ,$ and let $l(n,h)$
denote the number of visits to $h$ in the first $n$ steps:%
\begin{equation*}
l(n,h)=\#\{i:0\leq i\leq n,\text{ }X_{i}=h\}.
\end{equation*}%
Set also 
\begin{equation*}
l^{g}_*(n,h) =\left\{%
\begin{array}{ll}
l(n,h) & \mbox{ if } h\not\in \{ e_H,g\} \\ 
l(n,e_H)-1/2 & \mbox{ if } h=g \\ 
l(n,e_H)-1 & \mbox{ if } h=e_H.%
\end{array}
\right.
\end{equation*}

From \cite{Pittet2002}, probability that the random walk on $K\wr H$ driven
by $q$ is at $(h,g)\in K\wr H$ at time $n$ is given by 
\begin{equation*}
q^{(n)}((f,g))=\mathbf{E}\left( \prod\limits_{h\in H}\eta
^{(2l^g_*(n,h))}(f(h))\boldsymbol{1}_{\{X_{n}=g\} }\right)
\end{equation*}%
Note that $\mathbf{E}$ stands for expectation with respect to the random
walk $(X_i)_0^\infty$ on $H$ started at $e_H$.

From now on we assume that $\eta$ satisfies $\eta(e_K)=\epsilon >0$ so that 
\begin{equation*}
\epsilon \eta^{(n-1)}(e_K) \le \eta^{(n)}(e_K)\le \epsilon^{-1}
\eta^{(n-1))}(e_K).
\end{equation*}
Write $f\overset{C}{\asymp} g$ if $C^{-1}f\le g\le Cf$. Under these
circumstances, we have 
\begin{equation*}
q^{(n)}((\boldsymbol{e}_K,g)) \overset{1/\epsilon^{3}}{\asymp} \mathbf{E}%
\left( \prod\limits_{h\in H}\eta ^{(2l(n,h))}(e_K)\boldsymbol{1}%
_{\{X_{n}=g\} }\right)
\end{equation*}
so that we can essentially ignore the difference between $l$ and $l_*$.

Set%
\begin{equation*}
F_{K}(n):=-\log \eta ^{(2n)}(e_K)
\end{equation*}%
so that, for any $g\in H$, 
\begin{equation}  \label{trick}
q^{(n)}((\boldsymbol{e}_K,g))\simeq \mathbf{E}\left( e^{
-\sum_{H}F_{K}(l(n,h))}\boldsymbol{1}_{\{X_{n}=g\}}\right).
\end{equation}

\begin{pro}
\label{WreathLower} Let $H$ be a finitely generated group equipped with a
symmetric measure $\mu$ with $\mu(e_H)>0$. Let $K$ be a finitely generated
group equipped with a symmetric measure $\eta$ with $\eta(e_K)>0$ . Let $%
\|\cdot\|$ be a norm with volume function $V$. Let $r$ be a positive
continuous increasing function. Assume that:

\begin{enumerate}
\item The measure $\mu$ is strongly $(\|\cdot\|,r)$-controlled and $V$
satisfies $V(t)\simeq t^D$.

\item The function $r$ satisfies $r(t)= t^{1/\beta} \ell_1(t)$ where $\ell_1$
is a positive continuous  slowly varying function.

\item The function $F_K(n)= -\log \eta^{(2n)}(e_K)\simeq n^\gamma \ell_2(n)$
where $\gamma\in [0,1)$ and $\ell_2$ is a positive continuous slowly varying function.
\end{enumerate}

Assume also that the slowly varying functions $\ell_i$, $i=1,2$, are such
that $\ell_i(t^a)\simeq \ell_i(t)$ for all $a>0$. Then the
switch-walk-switch measure $q$ on $K\wr H$ associated with the pair $\eta$, $%
\mu$ satisfies 
\begin{equation*}
q ^{(n)}(e)\ge \exp \left( -C n^{\frac{D(1-\gamma)+\gamma\beta}{%
D(1-\gamma)+\beta}} \ell_1(n)^{\frac{\beta D(1-\gamma)}{D(1-\gamma)+\beta}}
\ell_2(n)^{\frac{\beta}{D(1-\gamma)+\beta}} \right).
\end{equation*}
\end{pro}

\begin{proof} 
Let $\mathfrak{m}$ be the spectral measure of $\eta^{(2)}$ in the sense that
\[
\int_{\lbrack 0,1]}t^{n}d\mathfrak{m}(t)=\nu ^{(2n)}(o). 
\]%
For $x\in [0,\infty)$, set \[
F(x):=-\log \int_{[0,1]}t^{x}d\mathfrak{m}(t).
\]%
Observe that $F_K(n)=F(n)$ and that $F$ is a concave function.  For $\tau>0$, let
 $B(\tau)=\{h\in H: \|h\|\le \tau\}$. Since $q^{(2n)}(e)\ge q^{(2n)}(x)$
for any $x\in K\wr H$, (\ref{trick}) yields
\begin{eqnarray*}
q^{(n)}(e) &\geq & \frac{1}{\#B(\tau)} \mathbf{E}\left( e^{
-\sum_{H}F_{K}(l(n,h))}\boldsymbol{1}_{\{\|X_n\|\le \tau\}}\right)\\
&\geq & \frac{1}{\#B(\tau)} \mathbf{E}\left( e^{
-\sum_{H}F_{K}(l(n,h))}\boldsymbol{1}_{\{\max_{1\le k\le n}\{\|X_k\|\}\le \tau\}}
\right).
\end{eqnarray*} 
Using the concavity of $F$ and the confinement of the walk $(X_n)$
on $H$ in the ball $B(\tau)$ in the last expression, this yields
$$q^{(n)}(e)\geq 
\frac{1}{V(\tau)}e^{ -V(\tau)F\left(n/V(\tau)\right) } 
\mathbf P_e\left(\max_{1\le k\le n}\{\|X_k\|\}\le \tau\right)
.$$
Let $\tau_n$ be such  that 
$$ V(\tau_n)F(n/V(\tau_n))= n/\rho(\tau_n)$$
where $\rho$ is the inverse of $r$. By our various assumption, this means
$$ (\tau_n^D/n)^{1-\gamma} \ell_2(n/\tau_n^D)
= \tau_n ^{-\beta}\ell_1(\tau_n)^{\beta}.
$$
Hence  $n\ra \tau_n$ is a regularly varying function 
of order $(1-\gamma)/(\beta+D(1-\gamma)) < 1/\beta$. 
This shows that $r(n)\gg \tau_n$ and,
since $\mu$ is strongly $(\|\cdot\|,r)$-controlled, Proposition \ref{pro-StC}
yields
$$q^{(n)}(e)\geq 
\frac{1}{V(\tau_n)}e^{ -V(\tau_n)F\left(n/V(\tau_n)\right) } 
e^{-C n/\rho(\tau_n)} \ge e^{-C_1 n/\rho(\tau_n)}
$$
and
$$\frac{n}{\rho(\tau_n)}=
n^{\frac{D(1-\gamma)+\gamma\beta}{D(1-\gamma)+\beta}} \ell_1(n)^{\frac{\beta
D(1-\gamma)}{D(1-\gamma)+\beta}}  \ell_2(n)^{\frac{\beta}{D(1-\gamma)+\beta}} .$$
This gives the stated lower bound on $q^{(n)}(e)$.
\end{proof}

\begin{rem}
The case $\gamma =1$ is excluded from this computations. It can be
treated by the same method but  $\tau _{n}$
become a slowly varying function of $n$.
\end{rem}

\begin{rem}
In the setting of Proposition \ref{WreathLower}, suppose in addition we have 
$H=\mathbb{Z}^{d}$ and $\mu $ is in the domain of attraction of an
operator-stable law $\nu $ on $\mathbb{R}^{d}$, that is there exists a
normalizing sequence $B_{n}\in GL_{d}(\mathbb{R})$ such that $B_{n}^{-1}\mu
^{\ast n}\Rightarrow \nu $. Then the lower bound in Proposition \ref%
{WreathLower} is sharp and agrees with \cite[Theorem 4.2]{SCZ-dv1}. Note
that in this case, $\det B_{n}\simeq V(r(n))$, the scaling relation in \cite[%
Theorem 4.2]{SCZ-dv1}  reads
\begin{equation*}
\frac{a_{n}\det B_{a_{n}}}{n}F_{K}\left( \frac{n}{\det B_{a_{n}}}\right)
\simeq 1
\end{equation*}%
and it agrees with $$ V(\tau_n)F(n/V(\tau_n))= n/\rho(\tau_n),$$
with $a_{n}\simeq
\rho (\tau _{n})$.
\end{rem}

\begin{exa}
Consider the symmetric probability measure $\mu $ on $\mathbb{Z}$ of the form 
\begin{equation*}
\mu (n)=\sum_{m\in \mathbb{Z}}\frac{\kappa \ell_1 (|n|)}{(1+|n|)^{1+\alpha }}
\end{equation*}%
where  and $\alpha \in (0,2)$ and $\ell_1 $ is a positive continuous 
slowly varying function satisfying $\ell_1
(t^{b})\simeq \ell_1 (t)$ for all $b>0$. We have
\begin{equation*}
\delta _{R}:=\sum_{\left\vert n\right\vert >R}\mu (n)\sim \frac{\kappa \ell_1
(R)}{\alpha R^{\alpha }}\mbox{ and }\;\;\mathcal{G}(R)=\sum_{\left\vert
n\right\vert \leq R}\left\vert n\right\vert ^{2}\mu (n)\sim \frac{\kappa }{%
2-\alpha }R^{2-\alpha }\ell_1(R).
\end{equation*}%
Therefore $R^{2}\delta _{R}/\mathcal{G}(R)\rightarrow (2-\alpha )/\alpha $.
By a classical result (see \cite{FelB5}), 
$\mu $ is in the domain of attraction of an $\alpha $-stable law on 
$\mathbb{R}$. The normalizing sequence $b_{n}$ can be chosen as the solution
to the equation $nb_{n}^{-2}\mathcal{G}(b_{n})=1$, that is $b_{n}\sim \left( 
\frac{\kappa }{2-\alpha }n\ell_1 (n)\right) ^{1/\alpha }$. Let $K$ be a finitely
generated group equipped with a symmetric measure $\eta $ with $\eta
(e_{K})>0$ . Suppose that the function $F_{K}(n)=-\log \eta
^{(2n)}(e_{K})\simeq n^{\gamma }\ell _{2}(n)$ where $\gamma \in \lbrack 0,1)$
and $\ell _{2}$ is a positive continuous  slowly varying function. Assume also that $%
\ell _{2}(t^{a})\simeq \ell _{2}(t)$ for all $a>0$. Then  \cite[Theorem 4.2]%
{SCZ-dv1} (and the remark following that statement in \cite{SCZ-dv1}) 
implies that the switch-walk-switch
measure $q$ on $K\wr H$ associated with the pair $\eta $, $\mu $ satisfies 
\begin{equation*}
-\log q^{(n)}(e)\simeq n/a_{n}\simeq n^{\frac{(1-\gamma )+\gamma \beta }{%
(1-\gamma )+\beta }}\ell_1 (n)^{\frac{(1-\gamma )}{(1-\gamma )+\beta }}\ell
_{2}(n)^{\frac{\beta }{(1-\gamma )+\beta }},
\end{equation*}%
where $a_{n}$ is computed from the scaling relation%
\begin{equation*}
\frac{a_{n}b_{a_{n}}}{n}F_{K}\left( \frac{n}{b_{a_{n}}}\right) \simeq 1.
\end{equation*}%
This agrees with the lower bound in Proposition \ref{WreathLower}. Note that,
by Proposition \ref{pro-SC}, the measure $\mu $ is strongly $(\left\vert
\cdot \right\vert ,r)$-controlled where $r(n)=\left( n\ell_1(n)\right) ^{1/\alpha
}$.
\end{exa}

\noindent{\bf Acknowledgments:} The authors thank Mathav Murugan for his comments and
 useful remarks.


\providecommand{\bysame}{\leavevmode\hbox to3em{\hrulefill}\thinspace}
\providecommand{\MR}{\relax\ifhmode\unskip\space\fi MR }
\providecommand{\MRhref}[2]{%
  \href{http://www.ams.org/mathscinet-getitem?mr=#1}{#2}
}
\providecommand{\href}[2]{#2}

\end{document}